\documentclass[11pt,a4paper]{article}

\makeatletter
\let\@fnsymbol\@arabic
\makeatother
\usepackage{graphicx}
\usepackage{amsmath}
\usepackage{amsfonts}
\usepackage{amssymb}
\usepackage{comment}
\usepackage{enumerate}
\usepackage{lscape}
\usepackage{longtable}
\usepackage{rotating}
\usepackage{color}
\usepackage{url}
\usepackage{subfigure}
\usepackage{rotating}
\usepackage{algorithmic}
\usepackage{titlesec}
\usepackage[titletoc]{appendix}
\usepackage{adjustbox}
\usepackage{mathrsfs}

\usepackage{color}
\usepackage{multirow}
\usepackage{diagbox}
\usepackage{appendix}

\newtheorem{theorem}{Theorem}[section]

\newtheorem{corollary}{Corollary}[section]

\newtheorem{remark}{Remark}[section]

\newenvironment{proof}[1][Proof]{\textbf{#1.} }{\hfill$\Box$}

\def\real{\mathbb{R}}

\usepackage{floatrow}
\newfloatcommand{capbtabbox}{table}[][\FBwidth]

\begin{document}

\title{Modeling Hessian-vector products in nonlinear optimization: \\
New Hessian-free methods}

\author{L. Song\thanks{CMUC, Department of Mathematics, University of Coimbra,
3001-501 Coimbra, Portugal ({\tt lili.song@mat.uc.pt}).  Support for
this author was partially provided by FCT/Portugal under grants PD/BI/128105/2016
and UID/MAT/00324/2019.}
\and
L. N. Vicente\thanks{Department of Industrial and Systems Engineering,
Lehigh University,
200 West Packer Avenue, Bethlehem, PA 18015-1582, USA and
Centre for Mathematics of the University of Coimbra (CMUC)
({\tt lnv@lehigh.edu}). Support for
this author was partially provided by FCT/Portugal under grants
UID/MAT/00324/2019 and P2020 SAICTPAC/0011/2015.}
}

\maketitle
\footnotesep=0.4cm
{\small
\begin{abstract}
In this paper, we suggest two ways of calculating interpolation models for unconstrained smooth nonlinear optimization when Hessian-vector
products are available. The main idea is to interpolate the objective function using a quadratic on a set of points around the current one and concurrently using the curvature information from products of the Hessian times appropriate vectors, possibly defined by the interpolating points. These enriched interpolating conditions form then an affine space of model Hessians or model Newton directions, from which a particular one can be computed once an equilibrium or least secant principle is defined.

A first approach consists of recovering the Hessian matrix satisfying the enriched interpolating conditions, from which then a Newton direction model can be computed. In a second approach we pose the recovery problem directly in the Newton direction. These techniques can lead to a significant reduction in the overall number of Hessian-vector products when compared to the inexact or truncated Newton method, although simple implementations may pay a cost in linear algebra or number of function evaluations.
\end{abstract}

\vspace{0.5ex}

\begin{center}
\textbf{Keywords:} Nonlinear/Nonconvex Optimization, Hessian-Vector Products, Quadratic Interpolation, Newton Direction, Hessian Recovery.
\end{center}
}

\section{Introduction}

Let us consider the minimization of a twice continuously differentiable
function $f$,
\[
\min_{x \in \real^n} \; f(x),
\]
in a context where the following information is available: Given $x \in \real^n$, one can compute $f(x)$,
$\nabla f(x)$, and $\nabla^2 f(x) v$ for any vector $v \in \real^n$.

\subsection{Literature review}

Newton-based methods for unconstrained nonlinear optimization require
the solution of a linear system at each iteration. The matrix of this system
is the Hessian of~$f$ and its right-hand side is the negative of the gradient.
There are instances where the Hessian is not available for factorization or
where that is too expensive but where one can afford to do Hessian-vector products,
in which cases the system should not be solved directly but by an iterative method. Then it is known that there is a residual error in the application of
the iterative solver and that such a residual can be made smaller by asking more from the solver.
This reasoning gave rise to the so-called inexact or truncated Newton methods
which have formed an important numerical tool for many decades.
It is well known since the contribution~\cite{RSDembo_SCEisenstat_TSteihaug_1982}
what conditions one should impose on the norm of the residual of the linear system
to obtain linear,
superlinear, or quadratic local convergence in the iterates of the underlying method
(see~\cite{JNocedal_SJWright_2006}).
Global convergence of inexact Newton methods
is also well studied~\cite{SCEisenstat_HFWalker_1994a,LGrippo_FLampariello_SLucidi_1986}.
One knows well also how to deal with negative curvature while solving the linear system
using Krylov-type methods (Conjugate Gradients or Lanczos), either using a trust-region technique~\cite{NIMGould_SLucidi_MRoma_PhLToint_1999,TSteihaug_1983a}
or a line search~\cite{SGNash_1984}.

When Hessian or Hessian-vector products are not available, estimating the Hessian can then play an important role, however the existing approaches are not entirely satisfactory. If the Hessian matrix is sparse and its sparsity pattern is known, the approach in~\cite{RFletcher_AGrothey_SLeyffer_1997} enforces multiple secant equations in a least squares sense, solving then a positive semi-definite system of equations in the nonzero Hessian components.
Their approach does not show a significant improvement compared to the L-LBGS or Newton trust-region methods. In~\cite{MJDPowell_PLToint_1979} the Hessian is estimated by finite differences in the gradient, but by dividing the Hessian columns first into groups. Using symmetry and the known sparsity of the Hessian, it is possible to find approximations to different Hessian columns at once. This method is cheap in computer arithmetic and provided better results when compared to~\cite{ARCurtis_MJDPowell_JKReid_1974}.
A more recent approach~\cite{CCartis_etal_2018} imposes the secant equations componentwise, leading to fewer equations when taking into account the available sparsity pattern.
The numerical results show that the algorithm can find the Hessian approximation fast and accurately when the number of nonzero entries per row is relatively low.

\subsection{The contribution of the paper}

In this paper two techniques are proposed and analyzed for the Hessian-free scenario where only Hessian-vector products are available for use. Our goal is to use as few of these products as possible without losing the ability to converge to a solution or a stationary point of the original problem.
Having this in mind we form a quadratic model around a point~$x$, using function and gradient values at~$x$ and function values at the interpolating points~$y^\ell$, $\ell=1,\ldots,p$. The matrix~$H$ of this model or some kind of Newton step has then to be recovered.

Our first approach enriches these interpolating conditions with the information coming from a single true Hessian-vector product $\nabla^2 f(x) (y-x)$, for a point~$y$ different from any of the~$y^\ell$'s of those conditions. In fact, to avoid degeneracy in the enriched interpolating conditions (which are affine conditions on~$H$), one has to choose~$y$ differently from those~$y^\ell$'s and one cannot consider more than one of these products.
The computation of the model Hessian is carried out by minimizing its norm or its distance to a previous model Hessian (say
from a previous iteration of the optimization method)
subject to the enriched interpolating conditions.
Such a Hessian recovery can then lead to the computation of an approximate Newton step.

Our second approach allows us to consider more than one Hessian-vector product in the model formulation.
The interpolating conditions are now enriched by the second-order
information coming from the Hessian-vector products $\nabla^2 f(x) (y^\ell-x)$, $\ell=1,\ldots,p$. Then, avoiding degeneracy and the inverse of the Hessian model, the recovery is done in the space of the Newton direction models, using a modified set of enriched interpolating conditions.
Again, the computation of the Newton direction model is carried out by minimizing its norm or its distance to a previous Newton direction model
subject to the modified enriched interpolating conditions.

In both cases we will provide some theoretical support for the recoveries by proving that the absolute error (in model Hessian or in model Newton direction) is decreasing in the case where~$f$ is quadratic and the enriched interpolating conditions are underdetermined. The recovery absolute error coming from the enriched interpolating conditions (in a determined situation) will be also analyzed for both cases. We report numerical results to confirm that both approaches are sound and can lead to a significant reduction in the number of Hessian-vector products.
The dimension of the problems tested is rather small. The linear algebra is dense, and the number of functions evaluations used can be relatively high. It is left for future research the application to medium/large-scale problems. The second approach based on a Newton direction model can be easily parallelized (see Section~\ref{sec:finalremarks}).

The paper is organized as follows. In Section~\ref{sec:Hrecovery} we present our first approach, the one for the recovery of a model Hessian.
In Section~\ref{sec:Hinverserecovery} we describe our second approach,
the one for the recovery of a model Newton direction.
In both cases we report illustrative numerical results for small problems.
The paper is finished in Section~\ref{sec:finalremarks}
with some final remarks and prospects of future work.
The notation $\mathcal{O}(A)$ will be used to represent the product of a constant times $A$ whenever
the multiplicative is independent of~$A$.
All vector and matrix norms are Euclidian unless otherwise specified.

\section{Hessian recovery from Hessian-vector products} \label{sec:Hrecovery}

Let $x$ be a given point. Suppose also that we have calculated $f$ and $\nabla f$ at $x$ as well as $f$ at a number of points $y^1,\ldots,y^p$. We can then use quadratic interpolation to fit the data by determining a symmetric matrix~$H$ such that
\begin{equation} \label{interpolation}
f(x) + \nabla f(x)^\top (y^\ell - x) + \frac{1}{2} (y^\ell - x)^\top H (y^\ell-x) \; = \; f(y^\ell),
\quad \ell=1,\ldots,p.
\end{equation}
Furthermore, given a set of vectors $v^1,\ldots,v^m$, with~$m$ possibly much smaller than~$n$, suppose that we have calculated
$w^j=\nabla^2 f(x) v^j$, $j=1,\ldots,q$. Hence we could then ask our Hessian model~$H$ to satisfy $H v^j = w^j$, $j=1,\ldots,q$.
However it is important to notice two immediate facts, reported in Remark~\ref{rem:LI}.

\begin{remark} \label{rem:LI}
First we cannot have $q > 1$. Any use of a pair $v^1,v^2$ would make the conditions $H v^1=w^1$ and $H v^2=w^2$ degenerate in~$H$, in the sense that the matrix multiplying the component variables of~$H$ would be rank deficient. This fact can be easily confirmed from multiplying each by the other vector, i.e., by looking at $(v^2)^\top H v^1=(v^2)^\top w^1$ and $(v^1)^\top H v^2=(v^1)^\top w^2$.
Secondly, even when taking $q=1$, one cannot consider $v^1 = y^\ell-x$, for any $\ell$, for the exact same reason. In fact, multiplying $H(y^\ell-x)=w^1$ on the left by $(1/2)(y^\ell-x)^\top$ would lead us to the same term in $H$ as of the corresponding interpolating one in~(\ref{interpolation}).
\end{remark}

\subsection{Hessian recovery}
From Remark~\ref{rem:LI}, we know that we can only consider one vector~$v$ for the Hessian multiplication $w = \nabla^2 f(x)\, v$, and that this vector cannot be any of the interpolation vectors $y^\ell-x$.
Then, in the same vein as it was done in~\cite{ARConn_KScheinberg_LNVicente_2009}
for derivative-free optimization, a model Hessian~$H$ could then be calculated from the solution of the recovery problem
\begin{equation} \label{Hessian-model-original}
\min_{H} \quad \mbox{norm}(H)
\quad \mbox{s.t.} \quad (\ref{interpolation}) \;\; \mbox{and} \;\; H v = w.
\end{equation}
The $\mbox{norm}(H)$ could be taken in a certain $\ell_1$ sense, leading to a linear program (see~\cite{ASBandeira_KScheinberg_LNVicente_2012}). It could also be set as the Frobenius norm, $\mbox{norm}(H) = \| H \|_F$, leading to a quadratic program. Alternatively, one can recover a model Hessian in a least secant fashion (as done in~\cite{MJDPowell_2004} for derivative-free optimization using the Frobenius norm)
\begin{equation} \label{Hessian-model-prev}
\min_{H} \quad \mbox{norm}(H-H^{prev})
\quad \mbox{s.t.} \quad (\ref{interpolation}) \;\; \mbox{and} \;\; H v = w,
\end{equation}
where $H^{prev}$ is a previously computed model Hessian (say, from a previous iteration of an optimization scheme).

\subsection{Theoretical motivation}
We will now see that when $f$ is quadratic the error
in the difference between the optimal solution $H^*$ of (\ref{Hessian-model-prev}) and the true Hessian decreases relatively to the previous estimate $H^{prev}$.
To prove such a result it is convenient to use the Frobenius norm in~(\ref{Hessian-model-prev})
and consider:
\begin{equation} \label{recovery-problem-H}
\min_{H} \quad \frac{1}{2} \| H-H^{prev} \|^2_F
\quad \mbox{s.t.} \quad (\ref{interpolation}) \;\; \mbox{and} \;\; H v = w,
\end{equation}
Let us first write the quadratic $f$ centered at~$x$
\begin{equation} \label{f-quadratic}
f(y) \; = \; a + b^\top (y - x) + \frac{1}{2} (y - x)^\top C (y - x),
\end{equation}
where $a=f(x)$, $b=\nabla f(x)$, and $C$ is a symmetric matrix.

\begin{theorem} \label{th:f-quadratic-H}
Let $f$ be given by (\ref{f-quadratic}) and assume that
the system of linear equations defined by~(\ref{interpolation}) and
$H v = w$ is feasible and underdetermined in $H$.
Let $H^*$ be the optimal solution of problem (\ref{recovery-problem-H}).
Then
\begin{equation*}
 \| H^*-C \|^2_F \; \leq \; \| H^{prev}-C \|^2_F.
\end{equation*}
\end{theorem}

\begin{proof}
The proof follows the argument in~\cite{MJDPowell_2004}.
From~(\ref{interpolation}),
we have $(y^\ell - x)^\top (C-H^*) (y^\ell - x) = 0$, $\ell=1,\ldots,p$.
We also have $(C-H^*) v = 0$.
Hence,
$C-H^*$ is a feasible direction for the affine space in $H$ defined
by~(\ref{interpolation}) and $H v = w$. It then turns out that the function
\[
m(\theta) \; = \; \frac{1}{2} \| (H^*-H^{prev}) + \theta(C-H^*) \|^2_F
\]
has a minimum at $\theta = 0$.
From the trace definition of the Frobenius norm
\[
m'(\theta) \; = \; \left[ (H^*-H^{prev}) + \theta(C-H^*) \right]^\top (C-H^*).
\]
Hence,
\[
(H^*-H^{prev})^\top (C-H^*) \; = \; 0,
\]
which then implies (given the symmetry of the matrices and considering only the
diagonal entries of the above matrix product)
\[
\sum_{i=1}^n \sum_{j=1}^n  (H^*_{ij}-H^{prev}_{ij}) (C_{ij}-H^*_{ij}) \; = \; 0.
\]

The rest of the proof requires the following calculations:
\begin{equation*}
\begin{aligned}
& \| H^{prev}- C \|_{F}^{2} -  \|H^* - H^{prev} \|_{F}^{2} - \| H^* - C  \|_{F}^{2}  \\
                   & =  \sum_{i=1}^{n} \sum_{j=1}^{n} [ (H^{prev}_{ij} - C_{ij})^2 - (H^*_{ij} -H^{prev}_{ij} )^2 - (H^*_{ij} -  C_{ij})^2] \\
                   & = \sum_{i=1}^{n} \sum_{j=1}^{n} [(H^{prev}_{ij} - C_{ij}+ H^*_{ij} -H^{prev}_{ij} ) ( H^{prev}_{ij} -C_{ij} -H^*_{ij} +H^{prev}_{ij}) - (H^*_{ij} -C_{ij})^2  ]  \\
                   & =  \sum_{i=1}^{n} \sum_{j=1}^{n}  [ (H^*_{ij}- C_{ij}) (2H^{prev}_{ij}- C_{ij} - H^*_{ij} - H^*_{ij} + C_{ij})]\\
                   & = 2 \sum_{i=1}^{n} \sum_{j=1}^{n}  [ (H^*_{ij} - C_{ij})(H^{prev}_{ij} - H^*_{ij} )] \; = \; 0.
\end{aligned}
\end{equation*}
Hence we have established that
\begin{equation*}
\begin{aligned}
\| H^* - C \|_{F}^{2}
& =   \| H^{prev}-  C \|_{F}^{2}  -   \|H^* -  H^{prev} \|_{F}^{2}\\
& \leq  \| H^{prev}- C \|_{F}^{2}.
\end{aligned}
\end{equation*}
\end{proof}
\vspace{1ex}

Let $\alpha$ represent the coefficients of $H$ in $(1/2) w^\top Hw$ in terms of the monomial basis. The quadratic components of this basis are of the form
$(1/2)w_i^2$, $i=1,\ldots,n$ and
 $w_i w_j$, $1\leq i < j \leq n$.
 So, we have $(1/2)h_{11} w_1^2 = \alpha_1 [(1/2) w_1^2],\ldots,(1/2)h_{nn} w_n^2 = \alpha_n [(1/2) w_n^2]$, $h_{12} w_1 w_2 = \alpha_{n+1} [w_1 w_2]$ and so on.
 The recovery problem~(\ref{recovery-problem-H}) can then be formulated approximately\footnote{The norm used in (\ref{recovery-problem-alphaH}) for $\alpha$ is a minor variation of the Frobenius norm of $H$.} as
\begin{equation}\label{recovery-problem-alphaH}
\min_{\alpha} \quad \frac{1}{2}\| \alpha - \alpha^{prev} \|^2
\quad \mbox{s.t.} \quad M \alpha \; = \; \delta,
\end{equation}
where
$$M = \left[ \begin{array}{ll} M^1 \\ M^2      \end{array}\right],  \quad  \delta =  \left[ \begin{array}{ll} \delta^1\\ \delta^2 \end{array}\right],
$$
$$
M^1 \alpha  = \left[ \begin{array}{ll}  \frac{1}{2} (y^1-x)^\top H (y^1-x)\\
\quad \quad \quad \quad \vdots   \\
\frac{1}{2} (y^p-x)^\top H (y^p-x)\end{array}\right],  \quad
M^2 \alpha =  H v,
$$
$$
\delta^1=  \left[ \begin{array}{c} f(y^1) - f(x) - \nabla f(x)^\top (y^1-x)\\
\vdots \\
f(y^p) - f(x)  - \nabla f(x)^\top (y^p-x)    \end{array}\right], \quad
\delta^2=  w.
$$

Another piece of motivation for this approach comes from the fact that
the enriched interpolating conditions defined by~(\ref{interpolation}) and
$H v = w$, once determined (i.e. with as many equations as variables), may
produce a model Hessian~$H$ that used together with~$\nabla f(x)$
can give rise to a fully quadratic model. Such a model has the same orders of accuracy as a Taylor-based model~\cite{ARConn_KScheinberg_LNVicente_2008} (see also~\cite{ARConn_KScheinberg_LNVicente_2009}).

\begin{theorem} \label{th:Hessian-error}
If $p$ is chosen such that $p+n= \frac{n^2+n}{2}$ and if $M$ is nonsingular, then the model Hessian~$H$ resulting from $M \alpha = \delta$ in~(\ref{recovery-problem-alphaH}) can give rise to a fully quadratic model, in other words, one has
\begin{equation*}
\| H -\nabla^2 f(x) \| \; =  \; \mathcal{O}(\Delta_y),
\end{equation*}
where $\Delta_{y}=\max _{1 \le \ell \le p}\left\|y^{\ell}-x\right\|$ and the constant multiplying $\Delta_{y}$ depends on the inverse of an appropriate scaled version of~$M$.
\end{theorem}

\begin{proof}
First we follow the argument in~\cite{ARConn_KScheinberg_LNVicente_2008}[Theorem~4.2] and consider
that~$x$ is at the origin, without any lost of generality.
One can start by making a Taylor expansion of~$f$ around~$x$ along all the displacements~$y^\ell-x$, $\ell=1,\ldots,p$, leading to
\begin{equation} \label{error-in-interpolation}
M^1 (\alpha- \alpha^x) \; = \; \mathcal{O}(\Delta_y^3),
\end{equation}
where $\alpha^x$ stores the components of~$\nabla^2 f(x)$ and
each component of the right-hand side is bounded by $(1/6) L_{\nabla^2 f} \|y^\ell - x \|^3$, with $L_{\nabla^2 f}$ the Lipschitz constant of~$\nabla^2 f$.
One also has
\[
M^2 (\alpha - \alpha^x) \; = \; 0.
\]
Now we divide each row of~(\ref{error-in-interpolation}) by $\Delta_y^2$. The proof is concluded by considering $[M^1/\Delta_y^2 ; M^2]$ as the scaled version of $M$ alluded in the statement of the result.
\end{proof}

\subsection{Numerical results for the determined case} \label{subsec:num-det}

As we have discussed in Theorem~\ref{th:Hessian-error}, if $p$ is chosen such that $p+n= \frac{n^2+n}{2}$ and if the matrix~$M$ is nonsingular and well conditioned, the model Hessian~$H$ resulting from $M \alpha = \delta$ in~(\ref{recovery-problem-alphaH}) becomes fully quadratic. The error between the Hessian model~$H$ and $\nabla^2 f(x)$ is then of the $\mathcal{O}(\Delta_y)$, where $\Delta_{y}=\max _{1 \le \ell \le p}\|y^{\ell}-x\|$.

In this section we will report some illustrative numerical results to confirm that an approach
built on such an Hessian model can lead to an economy of Hessian-vector products.
Our term of comparison will be the inexact Newton method (as described in~\cite[Section 7.1]{JNocedal_SJWright_2006}), where the system
$\nabla^2 f(x) d^{IN}= -\nabla f(x)$ is solved by applying a truncated linear conjugate (CG) method
(stopping once a direction of negative curvature is found or a relative error criterion is met).
In our case, after computing~$H$ from solving $M \alpha = \delta$ in~(\ref{recovery-problem-alphaH}), to compute our search direction~$d^{MH}$,
we apply the exact same truncated CG method to $H d^{MH}= -\nabla f(x)$ as in the inexact Newton method.
The computed directions $d^{IN}$ or $d^{MH}$ are necessarily descent in the sense of making an
acute angle with~$-\nabla f(x)$.

For both the inexact Newton method and our model Hessian approach,
a new iterate is of the form $x+ \alpha d$, where $d$ is given by $d^{IN}$ or $d^{MH}$ respectively.
The same
cubic interpolation line search~\cite[Section 2.4.2]{WSun_YYuan_2006} is used to compute the stepsizes~$\alpha^{IN}$ and $\alpha^{MH}$. In this line search, the objective function is approximated by a cubic polynomial with function values at three points and a derivative value at one point. The line search starts with a unit stepsize and terminates either successfully with a value~$\alpha$ satisfying a sufficient decrease condition for the function (of the form $f(x+\alpha d) \leq f(x)+c_{1} \alpha  \nabla f(x)^\top d$, with $c_1 =10^{-4}$) or unsuccessfully with a stepsize smaller than $10^{-10}$.

To form the model described in~(\ref{Hessian-model-original}) one needs $p$ interpolation points $y^1,\cdots,y^p$ and one vector~$v$ for Hessian multiplication. We have used the following scheme:
Before the initial iteration we have randomly generated
a set of $p$~points, $\{y^1,\cdots,y^p\}$, and a vector~$v$, in the unit ball $B(0; 1)$ centered at the origin.
Then, at each iteration $x_k$, the interpolation points used were of the form $x_k+r_k y^\ell_k$,
$\ell=1,\ldots,p$, and the vector $v_k$ of the form $r_k\,v$, where $r_k = \min\{10^{-2}, \max\{10^{-4},\|x_{k}-x_{k-1}\| \} \}$, $k = 1,2,\ldots$.

For the purpose of this numerical illustration, we selected 48 unconstrained (smooth and nonlinear) very small problems from the CUTEst collection (see Appendix~\ref{APP:test_problems}), also used in the papers~\cite{CCartis_etal_2018,SGratton_CWRoyer_LNVicente_2018}.
Both methods were stopped when an iterate $x_k$ was found such that $\| \nabla f(x_k)\| < 10^{-5}$.
We built performance profiles (see Appendix~\ref{APP:PP}) using as performance metric the numbers of Hessian-vector products and iterations (Figure~\ref{PP:PP_Hv_DF_new}) and the number of function evaluations (Figure~\ref{PP:PP_func_DF}). One can see that our approach can effectively lead to a significant reduction on the number of Hessian-vector products. We estimate that this reduction is approximately 50\% as both approaches take on average~2~CG inner iterations to compute a direction, and the number of main iterations is comparable. Of course, one has to pay a significant cost in number of function evaluations which is of the order of~$n^2$ per main iteration.

 %%%%%%%%%%%%%%%
 \begin{figure}[H]
 	\begin{minipage}[H]{0.5\linewidth}
 		\centering
 		\includegraphics[height=5.5cm,width=8.5cm]{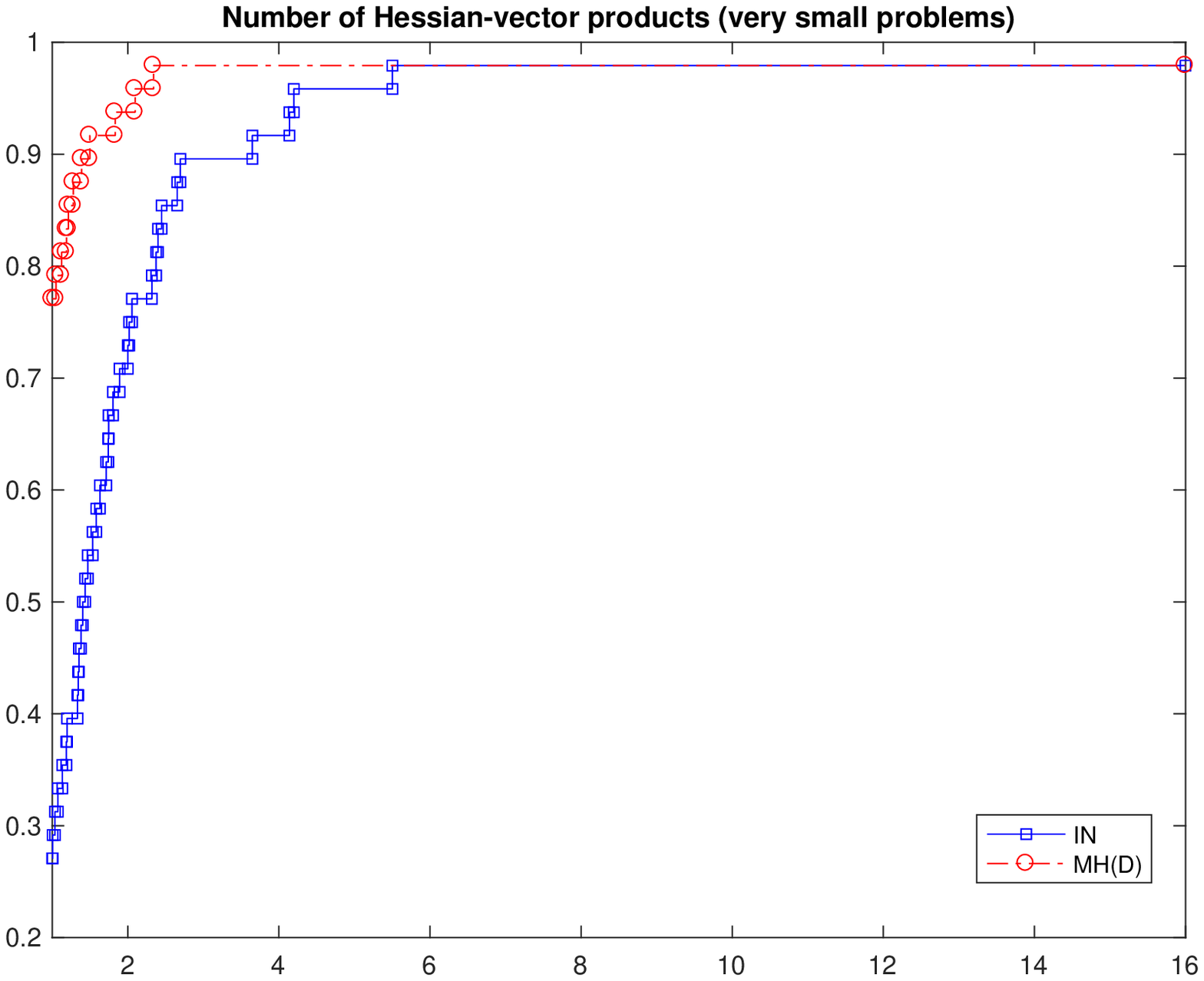}
 	\end{minipage}%
 	\begin{minipage}[H]{0.55\linewidth}
 		\centering
 		\includegraphics[height=5.5cm,width=8.5cm]{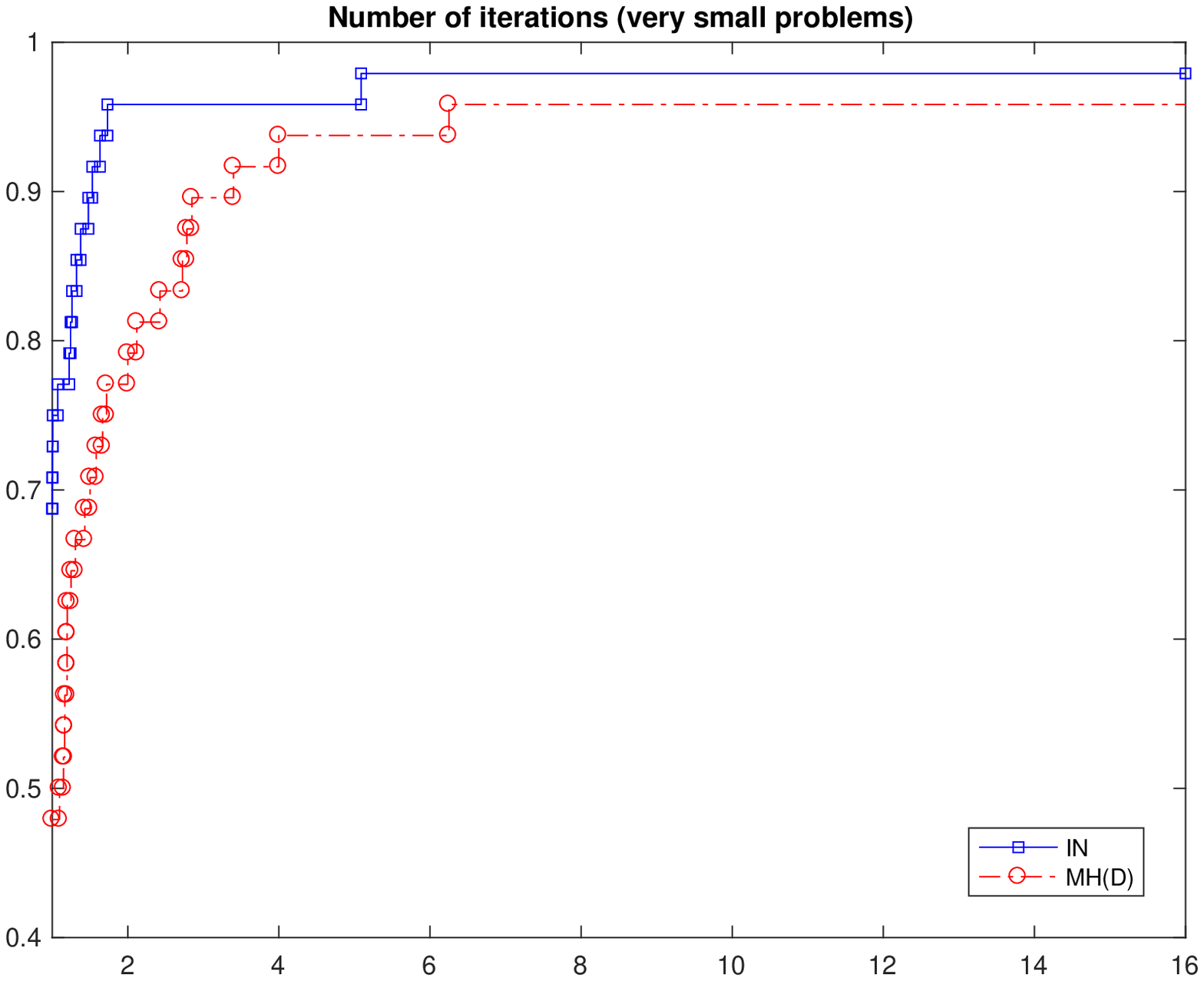}	
 	\end{minipage}
 \caption{Testing the Hessian recovery within a line-search algorithm. Performance profiles for the numbers of Hessian-vector products and iterations, for the set of very small problems of Appendix~\ref{APP:test_problems}. The value of $p$ was set to $\frac{n^2+n}{2}-n$.
 \label{PP:PP_Hv_DF_new}	}
 \end{figure}
 %%%%%%%%%%%%%%%
 		
 %%%%%%%%%%%%%%%
 \begin{figure}[H]
 	\centering
 	\includegraphics[height=5.5cm,width=8.5cm]{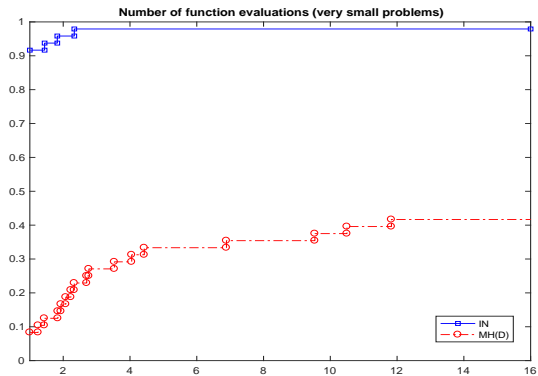}
 	\caption{Testing the Hessian recovery within a line-search algorithm.
 Performance profiles for number of function evaluations, for the set of very small problems of Appendix~\ref{APP:test_problems}. The value of $p$ was set to $\frac{n^2+n}{2}-n$.}
 	\label{PP:PP_func_DF}	
 \end{figure}
 %%%%%%%%%%%%%%%	

\subsection{Numerical results for the determined case when the Hessian sparsity is known}

In many optimization problems, the Hessian matrix of the objective function is sparse and the corresponding sparsity pattern is known.
Let $\Omega (\nabla^2f) = \{(i, j): i\leq j, \nabla^2f_{ij}(x) = 0 \text{ for all } x\}$ be the sparsity pattern of $\nabla^2 f$.
When $|\Omega (\nabla^2f)| \ll n(n+1)/2$, it is then beneficial and often necessary to use specialized algorithms and data structures that take advantage of the known sparsity pattern.
One can tailor our model Hessian approach to problems with sparse Hessian matrices when the  sparsity patterns are known.  We require the Hessian model to share the same sparsity pattern of the true Hessian, recovering only the nonzero
elements. In fact,
instead of solving problem~(\ref{recovery-problem-alphaH}) with respect to the whole Hessian matrix, we solve problem
\begin{equation}\label{recovery-problem-alphaH-sparse}
\min_{\alpha_{\Omega}} \quad \frac{1}{2}\| \alpha_{\Omega} - \alpha^{prev}_{\Omega}\|^2
\quad \mbox{s.t.} \quad M_{\Omega} \alpha_{\Omega} \; = \; \delta,
\end{equation}
where the elements in the rows of $M_{\Omega}$ and in the vector $\alpha_{\Omega}$ correspond now only to nonzero entries.

We have tested our sparse Hessian recovery approach using the same algorithmic environment of Subsection~\ref{subsec:num-det}, the only difference being in the usage of the model equation $M_{\Omega} \alpha_{\Omega} = \delta$ in~(\ref{recovery-problem-alphaH-sparse}) and a smaller value of $p$ (now given by the difference between the number of nonzeros of the Hessian and $n$, so that the matrix $M_{\Omega}$ is squared). The sparse problems used are listed in Appendix~\ref{APP:sparse_problems}.
The experiments are reported in Figures~\ref{PP:PP_hv_sparse} and~\ref{PP:PP_func_sparse} in the form of performance profiles. The conclusions are similar to those in Subsection~\ref{subsec:num-det}.

%%%%%%%%%%%%%%%
\begin{figure}[H]
\begin{minipage}[H]{0.5\linewidth}
\centering
\includegraphics[height=5.5cm,width=8.5cm]{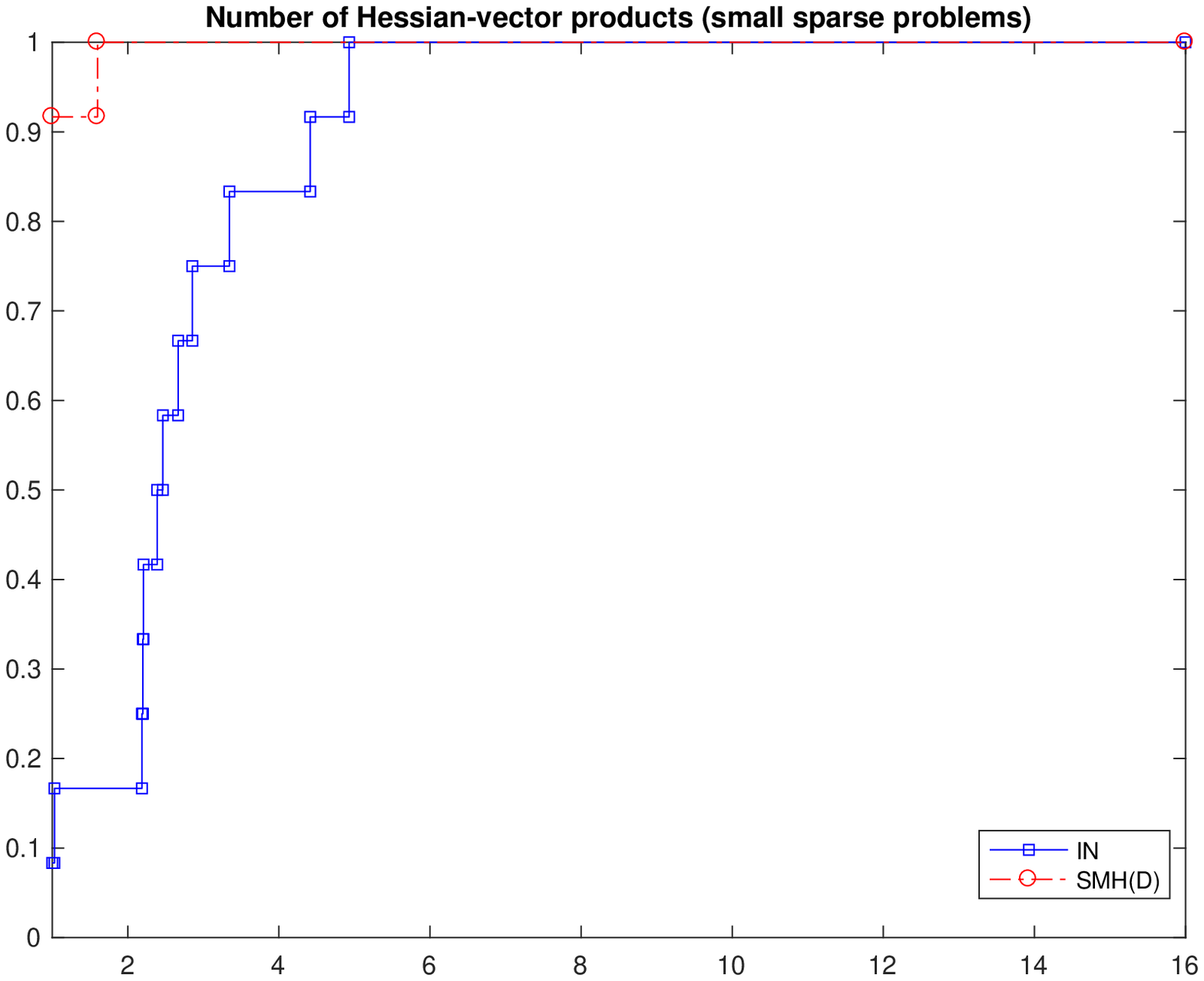}
\end{minipage}%
\begin{minipage}[H]{0.55\linewidth}
\centering
\includegraphics[height=5.5cm,width=8.5cm]{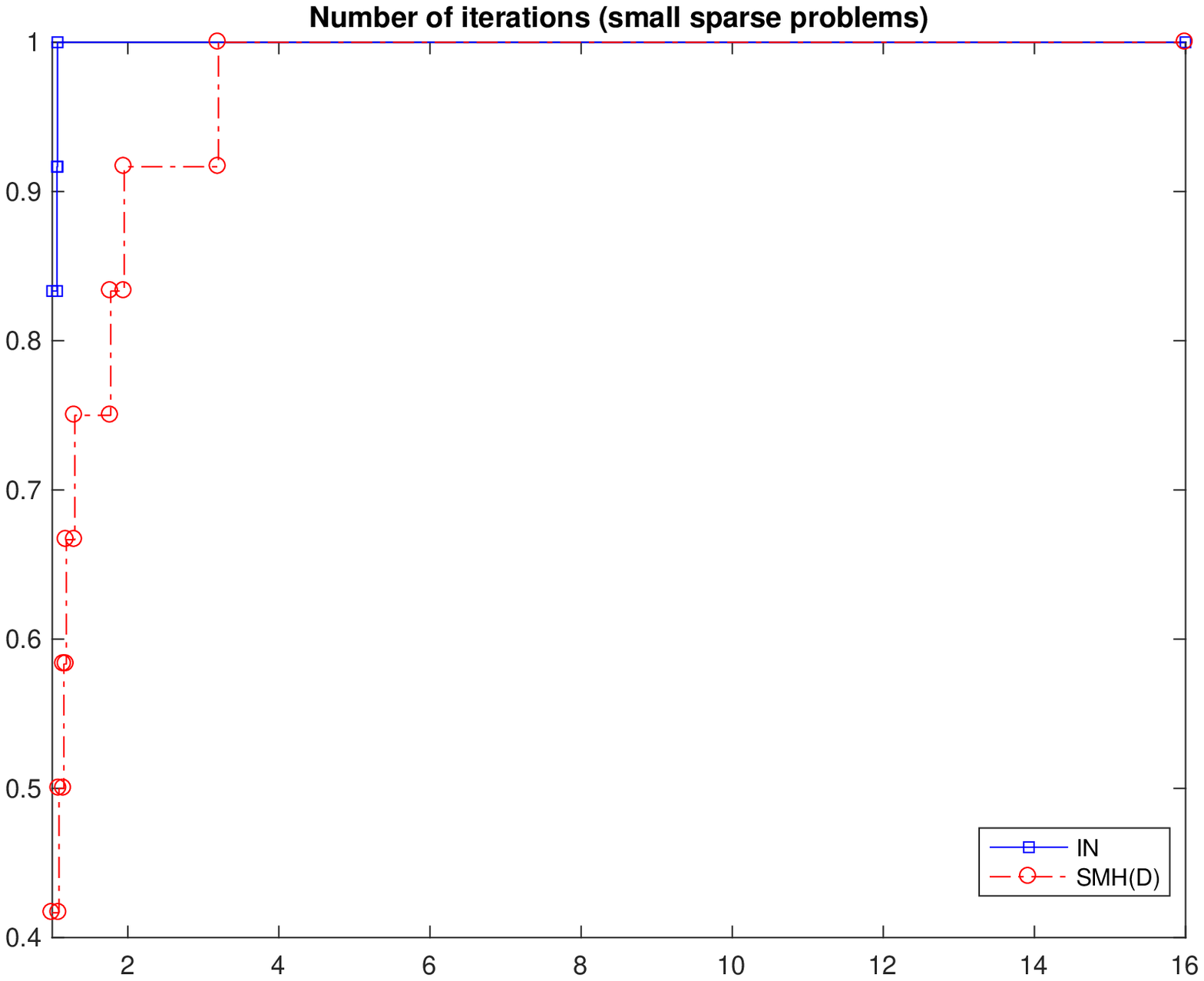}
\end{minipage}
\caption{Testing the Hessian recovery within a line-search algorithm. Performance profiles for the numbers of Hessian-vector products and iterations, for the set of small sparse problems of Appendix~\ref{APP:sparse_problems}. The value of $p$ was set to number of nonzeros minus $n$.}
\label{PP:PP_hv_sparse}
\end{figure}
%%%%%%%%%%%%%%%
\begin{figure}[H]
\centering
\includegraphics[height=5.5cm,width=8.5cm]{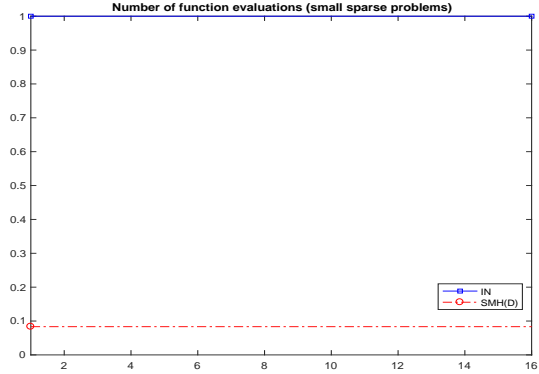}
\caption{Testing the Hessian recovery within a line-search algorithm. Performance profiles for number of function evaluations, for the set of small sparse problems of Appendix~\ref{APP:sparse_problems}. The value of $p$ was set to number of nonzeros minus $n$.}
\label{PP:PP_func_sparse}
\end{figure}

%%%%%%%%%%%%%%%

\subsection{Recovery cost in the general case}

The necessary and sufficient optimality conditions for the convex QP~(\ref{recovery-problem-alphaH}) can be stated as
\begin{equation}\label{necessary-sufficient-conditions}
\begin{aligned}
\alpha - \alpha^{prev} - M^\top \lambda  &=  0  \\
M \alpha  &=  \delta,
\end{aligned}
\end{equation}
where $\lambda$ denotes the Lagrange multipliers. Such multipliers can then be recovered by solving
\begin{equation}\label{linear-system-H}
M M^\top \lambda \; = \;  \delta - M\alpha^{prev} .
\end{equation}
The system~(\ref{linear-system-H}) can either be solved directly or iteratively.
If solved directly the cost is of the order of~$(p+n)^2n^2$ to form $M M^\top$ and  of~$(p+n)^3$ to factorize it, and the overall storage of the order of~$(p+n)^2$. If the Conjugate Gradient (CG) method is applied, the overall cost is
of the order of $c_g(p+n)n^2$, where~$c_g$ is the number of CG iterations. In fact,
each matrix vector multiplication with either $M^\top$ or $M$ costs $\mathcal{O}((p+n)n^2)$.
Solving the KKT system~(\ref{necessary-sufficient-conditions}) using an indefinite factorization is even less viable given that the storage space would be of the order of~$(n^2+p)^2$.

\section{Newton direction recovery from Hessian-vector products} \label{sec:Hinverserecovery}

In this section, we introduce a new approach to recover the Newton direction from Hessian-vector products that does not require an explicit recovery of the Hessian matrix.

\subsection{Newton direction recovery}

Let us first consider a quadratic Taylor expansion of the form
\begin{equation} \label{2nd}
f(x) + \nabla f(x)^\top (y^\ell - x) + \frac{1}{2} (y^\ell - x)^\top \nabla^2 f(x)
(y^\ell-x) \; \simeq \; f(y^\ell),
\quad \ell=1,\ldots,p,
\end{equation}
made using a sample set $\{ y^1,\ldots,y^p \}$.
We will synchronize expansion~(\ref{2nd})
with Hessian-vector products along $y^\ell-x$, $\ell=1,\ldots,p$.
In fact, we require the calculation of
\begin{equation} \label{Hessian-product}
z^\ell \; = \; \nabla^2 f(x) (y^\ell-x), \quad \ell=1,\ldots,p.
\end{equation}

Since our interest relies specifically on the calculation of the Newton direction,
assuming that the model Hessian~$\nabla^2 f(x)$ is nonsingular, we obtain from~(\ref{2nd})
and~(\ref{Hessian-product})
\begin{equation*}
f(x) + (\nabla^2 f(x)^{-1}\nabla f(x))^\top \nabla^2 f(x)(y^\ell-x) + \frac{1}{2} (y^\ell-x)^\top z^\ell \; \simeq \; f(y^\ell),
\quad \ell=1,\ldots,p.
\end{equation*}
Then, introducing the model vector $d \simeq -\nabla^2 f(x)^{-1}\nabla f(x)$, one arrives at a new set of enriched interpolating conditions
\begin{equation} \label{interpolation-inverse}
(z^\ell)^\top  d\; = \; -f(y^\ell) + f(x) + \frac{1}{2} (y^\ell-x)^\top z^\ell, \quad \ell=1,\ldots,p.
\end{equation}

Equations~(\ref{interpolation-inverse}) lead then to a new recovery problem
\begin{equation} \label{recovery-problem}
\min_{d} \quad \mbox{norm}(d-d^{prev})
\quad \mbox{s.t.}   \quad  (\ref{interpolation-inverse}).
\end{equation}
When $d^{prev}$ is the previously recovered Newton direction, we are following the spirit of a quasi-Newton least secant approach. One could also consider the case $d^{prev}=0$ as it was done in some derivative-free approaches for Hessian recovery.
Let us now give two arguments to motivate this approach.

\subsection{Theoretical motivation}

First, as in the previous section, we can provide motivation for this approach when~$f$ is assumed quadratic~(\ref{f-quadratic}), this time with a nonsingular Hessian~$C$. Here we need to consider the square of the $\ell_2$-norm in~(\ref{recovery-problem})
\begin{equation} \label{recovery-problem-F}
\min_{d} \quad \frac{1}{2} \|d-d^{prev}\|^2
\quad \mbox{s.t.}   \quad  (\ref{interpolation-inverse}).
\end{equation}
We will show that in the quadratic case the error in the approximation of the Newton direction
is monotonically non increasing.

\begin{theorem} \label{th:f-quadratic}
Let $f$ be given by (\ref{f-quadratic}) with $C$ nonsingular and assume that
the system of linear equations (\ref{interpolation-inverse}) is feasible and underdetermined in $d$.
Let $d^*$ be the optimal solution of problem~(\ref{recovery-problem-F}).
Then
\begin{equation}
 \| d^*-(-C^{-1}b) \|^2 \; \leq \; \| d^{prev}-(-C^{-1}b) \|^2.
\end{equation}
\end{theorem}

\begin{proof}
From the expression~(\ref{f-quadratic}) for $f$, one has
\[
f(y^\ell) \; = \; a + (C^{-1}b)^\top C(y^\ell - x) + \frac{1}{2} (y^\ell - x)^\top C (y^\ell - x), \quad \ell=1,\ldots,p.
\]
and hence, using $z^\ell = C (y^\ell-x)$, $\ell=1,\ldots,p$, and~(\ref{interpolation-inverse}),
one arrives at $(z^\ell)^\top (d^* - (-C^{-1}b) ) = 0$.
The conclusion is that
$d^* - (-C^{-1}b)$ is a feasible direction for the affine space in $d$ defined
by~(\ref{interpolation-inverse}).

The rest of the proof follows the same lines as in
the proof of Theorem~\ref{th:f-quadratic-H}.
The function
\[
m(\theta) \; = \; \frac{1}{2} \| (d^*-d^{prev}) + \theta(-C^{-1}b-d^*) \|^2
\]
has a minimum at $\theta = 0$, from which we conclude that $(d^*-d^{prev})^\top (-C^{-1}b-d^*) = 0$.
From here we obtain
\begin{equation*}
\begin{aligned}
\| d^* - (-C^{-1}b) \|^{2}
& =   \| d^{prev}-  (-C^{-1}b) \|^{2}  -   \|d^* -  d^{prev} \|^{2}\\
& \leq  \| d^{prev}- (-C^{-1}b) \|^{2}.
\end{aligned}
\end{equation*}
\end{proof}
\vspace{1ex}

The second argument establishes the accuracy of the recovery under the assumption that $p \geq n$
(see the end of this subsection for a discussion about this assumption and how to circumvent it practice).
We will establish a bound on the norm of the absolute error of the recovered Newton direction~$d^N$ based on
$\Delta_y = \max_{1 \leq \ell \leq p} \|y^\ell - x \|$,
$\Delta_z = \max_{1 \leq \ell \leq p} \|z^\ell \|$,
and the conditioning of the matrix $M_L^z$, whose rows are
$(1/\Delta_z)(z^\ell)^\top$, $\ell=1,\ldots,p$.

\begin{theorem} \label{error-bound}
Suppose that $p \geq n$, the matrix $M_L^z$ is full column rank,
and $\nabla^2 f(x)$ is invertible.
Then,
if $d^N$ satisfies (\ref{interpolation-inverse}), in a least squares sense when $p > n$, one has
\[
\left\| -\nabla^2 f(x)^{-1}\nabla f(x)- d^N \right\| \; \leq \;
\Lambda_z \mathcal{O}\left( \frac{ \Delta_y^3 }{ \Delta_z} \right),
\]
where $\Lambda_z$ is a bound on the norm of the left inverse of $M_L^z$ and
the multiplicative constant in $\mathcal{O}$ depends on the Lipschitz constant of $\nabla^2 f$.
\end{theorem}

\begin{proof}
Expanding $f$ at $y^\ell$ around $x$ in (\ref{interpolation-inverse}) yields
\begin{equation*}
 (-\nabla^2 f(x)^{-1}\nabla f(x)- d^N)^\top z^\ell \; = \;  \mathcal{O}(\Delta_y^3), \quad \ell=1,\ldots,p,
\end{equation*}
where the constant in $\mathcal{O}(\Delta_y^3)$ depends
on the Lipschitz constant of $\nabla^2 f$.
The result follows by dividing both terms by $\Delta_z$ and multiplying by the left inverse of $M_L^z$.
\end{proof}
\vspace{1ex}

One can derive an estimate solely dependent on $\Delta_y$ and
on the conditioning of the matrix $M_L^y$ formed by the rows
$(1/\Delta_y)(y^\ell-x)^\top$, $\ell=1,\ldots,p$.
In fact, from
\[
\Delta_y M_L^y \nabla^2 f(x) \; = \; \Delta_z M_L^z
\]
one has
\[
\left\| \left( M_L^z \right)^{\dagger} \right\|
\; = \; \frac{\Delta_z}{\Delta_y}
\| R_y \|,
\]
with
\begin{equation} \label{Ry}
R_y \; = \; \left( \nabla^2 f(x) (M_L^y)^\top (M_L^y) \nabla^2 f(x) \right)^{-1} \nabla^2 f(x) (M_L^y)^\top.
\end{equation}

\begin{corollary}   \label{error-bound-2}
Suppose that $p \geq n$, the matrix $M_L^z$ is full column rank,
and $\nabla^2 f(x)$ is invertible.
Then,
if $d^N$ satisfies (\ref{interpolation-inverse}), one has
\[
\left\| -\nabla^2 f(x)^{-1}\nabla f(x)- d^N \right\| \; \leq \;
\|R_y\| \mathcal{O}( \Delta_y^2 ),
\]
where
the multiplicative constant in $\mathcal{O}$ depends on the Lipschitz constant of $\nabla^2 f$.
\end{corollary}

Hence by controlling the geometry of the points $y^\ell$, $\ell=1,\ldots,p$, around $x$
one can provide an accurate bound when the Hessian of~$f$ is invertible and $p \ge n$. In general, we can attempt to control the conditioning of~$M_L^z$, replacing some of the points $y^\ell$
if necessary. Such a conditioning must eventually become adequate if the vectors $y^\ell-x$
are sufficiently linearly independent and
lie in eigenspaces of $\nabla^2 f(x)$ corresponding to eigenvalues not too close to zero.

Using $p=n$ Hessian-vector products at each iteration is certainly not a desirable strategy
as that would be equivalent to access the entire Hessian matrix. It is however possible to use $p \ll n$ and still obtain an accurate Newton direction model. The possibility we have in mind is to build upon a
previously computed Newton direction model calculated using $p=n$. Let $x_{prev}$ be such
an iterate, $y^1_{prev},\ldots,y^n_{prev}$ be the corresponding sample points, and
$z^1_{prev},\ldots,z^n_{prev}$ be the corresponding Hessian-vector products. Suppose we are now at a new iterate $x$ and we would like to reuse $f(y^1_{prev}),\ldots,f(y^n_{prev})$ and $z^1_{prev}=\nabla^2 f(x_{prev})(y^1_{prev}-x_{prev}),\ldots,z^n_{prev}=\nabla^2 f(x_{prev})(y^n_{prev}-x_{prev})$.
In such a case what we will have in~(\ref{interpolation-inverse}) is
\[
z^{\ell}_{prev} \; = \;
\nabla^2 f(x_{prev})(y^{\ell}_{prev}-x_{prev}) \; \simeq \; \nabla f(y^{\ell}_{prev}) - \nabla f(x_{prev}),
\quad \ell=1,\ldots,p,
\]
but what we wish we would have is
\[
z^{\ell} \; = \;
\nabla^2 f(x)(y^{\ell}_{prev}-x) \; \simeq \; \nabla f(y^{\ell}_{prev}) - \nabla f(x),
\quad \ell=1,\ldots,p,
\]
So, one can obtain an approximation to $z^{\ell}$ from
\begin{equation} \label{with-correction}
z^{\ell}_{prev} + \nabla f(x_{prev}) - \nabla f(x), \quad \ell=1,\ldots,p.
\end{equation}
The error in such an approximation is of the $\mathcal{O}(\max\{\| y^{\ell}_{prev}-x_{prev}\|^2,
\| y^{\ell}_{prev}-x\|^2 \})$, which would then has to be divided by $\Delta_z$ in the context of
Theorem~\ref{error-bound}. Of course, if we then keep applying this strategy the error will accumulate over the iterations, but there are certainly remedies such as bringing a few new, fresh $z$'s at each iteration and applying restarts with $p=n$ whenever the conditioning of $M^z_L$ becomes large.

\subsection{Numerical results for the determined case using a correction} \label{subsec:3num-det}

To use as few Hessian-vector products as possible, we start by using $p=n$ products at iteration zero, to then replace only one interpolation point at each iteration. We choose to replace the point farthest away from the current iterate~$x$. (A perhaps more sound approach would have been to choose the $z^\ell$ that has contributed the most to the conditioning of~$M_L^z$.) A new point is then added, generated in the ball $\mathcal{B}(x, r)$, where $r = \min\{10^{-2}, \max\{10^{-4}, \|x-x_{prev}\| \} \}$.
Therefore, only one more Hessian-vector product and one more function evaluation is required at each iteration. We then replace all other $z^{\ell}_{prev}$'s by~(\ref{with-correction}).
We monitor the condition number of $M^z_L$, and apply a restart (with $p=n$ as in iteration~$0$) whenever $\mbox{cond}(M^z_L) \geq 10^{8}$.

A Newton direction model~$d^N$ is then calculated by
solving~(\ref{interpolation-inverse}) directly.
To guarantee that we have a descent direction~$d$, meaning that $-\nabla f(x)^{\top} d > 0$, we modify the $d^N$ from~(\ref{interpolation-inverse}) so that
 $d = d^N - \beta \nabla f(x)$ where $\beta$ is such that $\cos( d, -\nabla f(x)) = \eta$, and $\eta$ was set to 0.95.

The modified Newton direction model was then used in a line-search algorithm using
the same cubic line search procedure of Subsection~\ref{subsec:num-det}. The comparison is again against the inexact Newton method (as described in~\cite[Section 7.1]{JNocedal_SJWright_2006}).
First we tested the very small problems of Appendix~\ref{APP:test_problems}. Again, we plot performance profiles (see Appendix~\ref{APP:PP}) using as performance metric the numbers of Hessian-vector products and iterations (Figure~\ref{PP:PP_hv_small}) and the number of function evaluations (Figure~\ref{PP:PP_func_small}). The results are quite encouraging.
We then selected a benchmark of 26 unconstrained nonlinear small problems from the CUTEst collection~\cite{NIMGould_DOrban_PhLToint_2015}, listed in Appendix~\ref{APP:larger_problems}. The experiments are reported in Figures~\ref{PP:PP_hv_larger} and~\ref{PP:PP_func_larger} in the form of the same performance profiles. The results are similar and again promising.

%%%%%%%%%%%%%%%
\begin{figure}[H]
	\begin{minipage}[H]{0.5\linewidth}
	%	\centering
		\includegraphics[height=5.5cm,width=8.5cm]{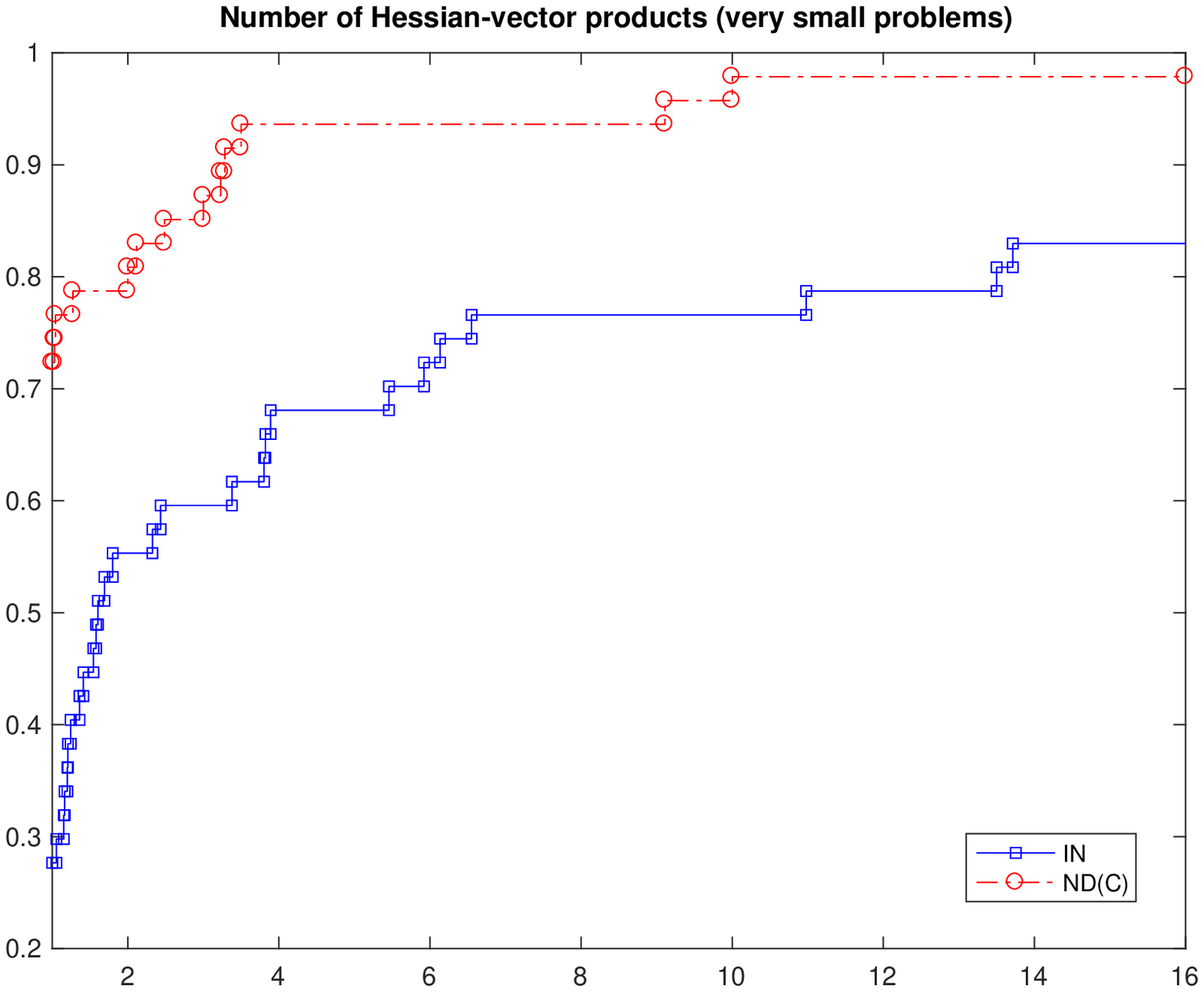}
	\end{minipage}%
	\begin{minipage}[H]{0.55\linewidth}
	%	\centering
		\includegraphics[height=5.5cm,width=8.5cm]{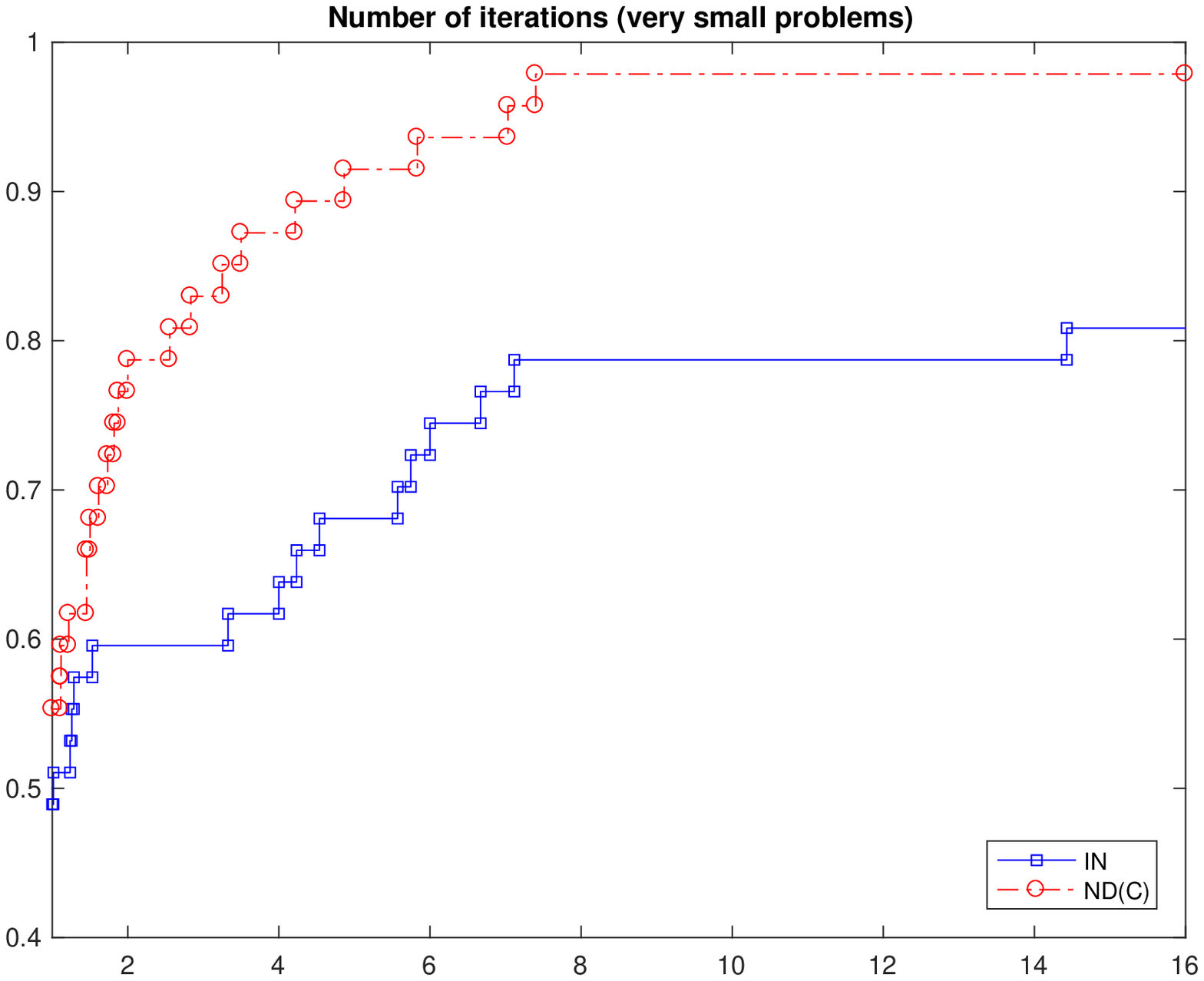}
	\end{minipage}
	\caption{Testing the Newton direction recovery within a line-search algorithm. Performance profiles for the numbers of Hessian-vector products and iterations, for the set of very small problems of Appendix~\ref{APP:test_problems}.}
	\label{PP:PP_hv_small}
\end{figure}
%%%%%%%%%%%%%%%
\begin{figure}[H]
	\begin{minipage}[H]{0.5\linewidth}
		\centering
		\includegraphics[height=5.5cm,width=8.5cm]{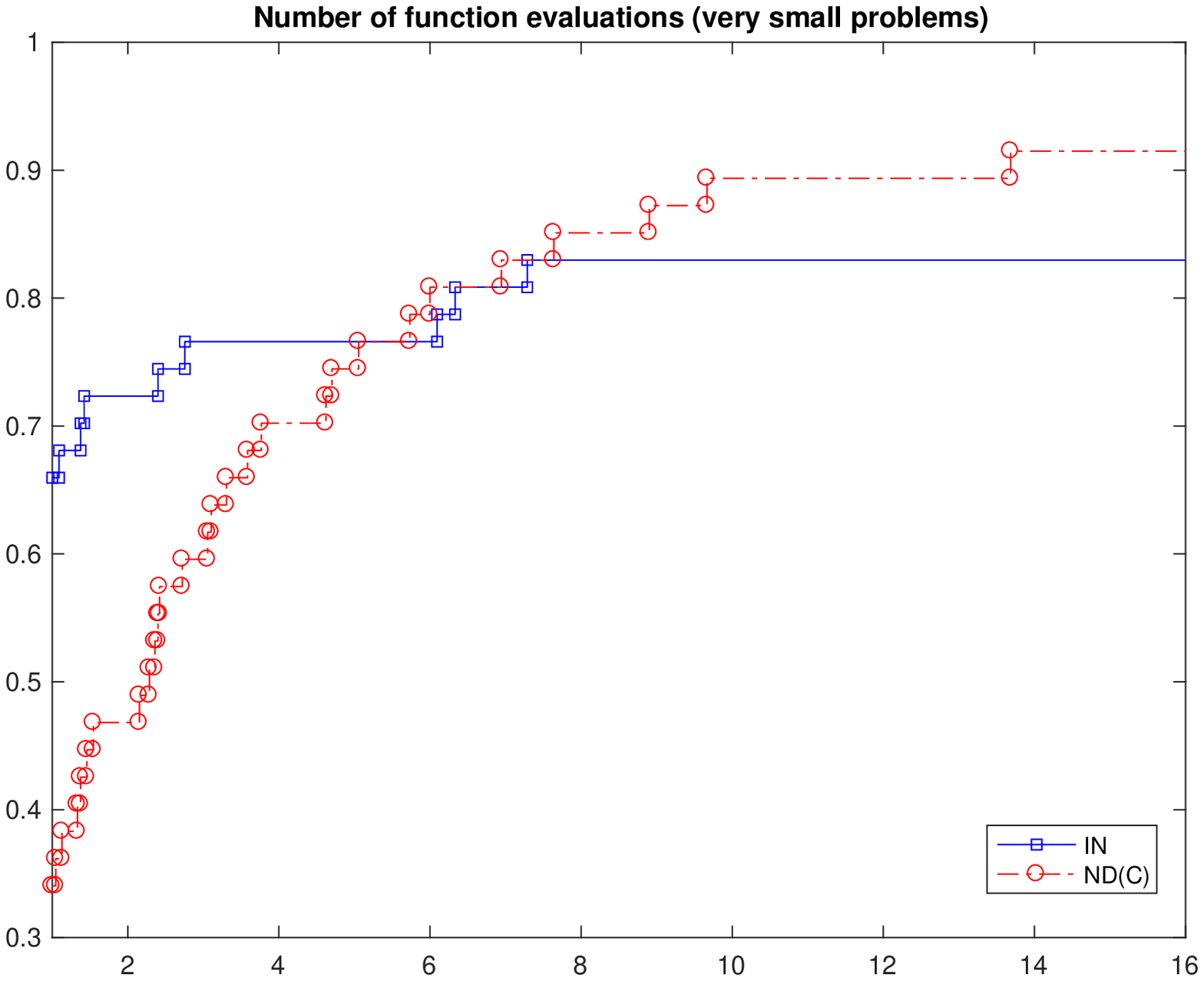}
	\end{minipage}%
%	\begin{minipage}[H]{0.55\linewidth}
%		\centering
%		\includegraphics[height=5.5cm,width=8.5cm]{PP_ND_CPU.eps}
%	\end{minipage}
	\caption{Testing the Newton direction recovery within a line-search algorithm. Performance profiles for the numbers of function evaluations for the set of very small problems of Appendix~\ref{APP:test_problems}.}
	\label{PP:PP_func_small}
\end{figure}
%%%%%%%%%%%%%%%

%%%%%%%%%%%%%%%
\begin{figure}[H]
	\begin{minipage}[H]{0.5\linewidth}
		%	\centering
		\includegraphics[height=5.5cm,width=8.5cm]{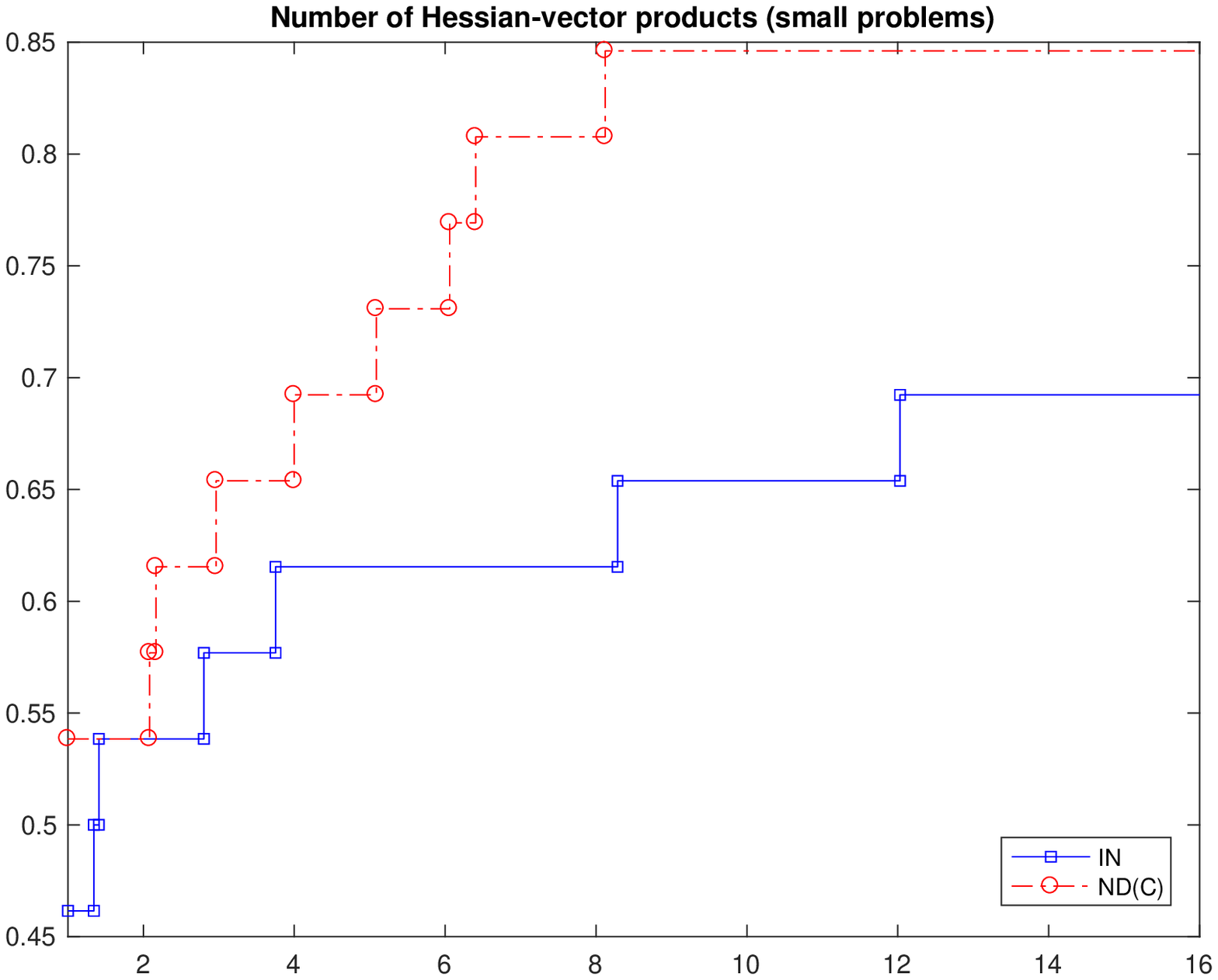}
	\end{minipage}%
	\begin{minipage}[H]{0.55\linewidth}
		%	\centering
		\includegraphics[height=5.5cm,width=8.5cm]{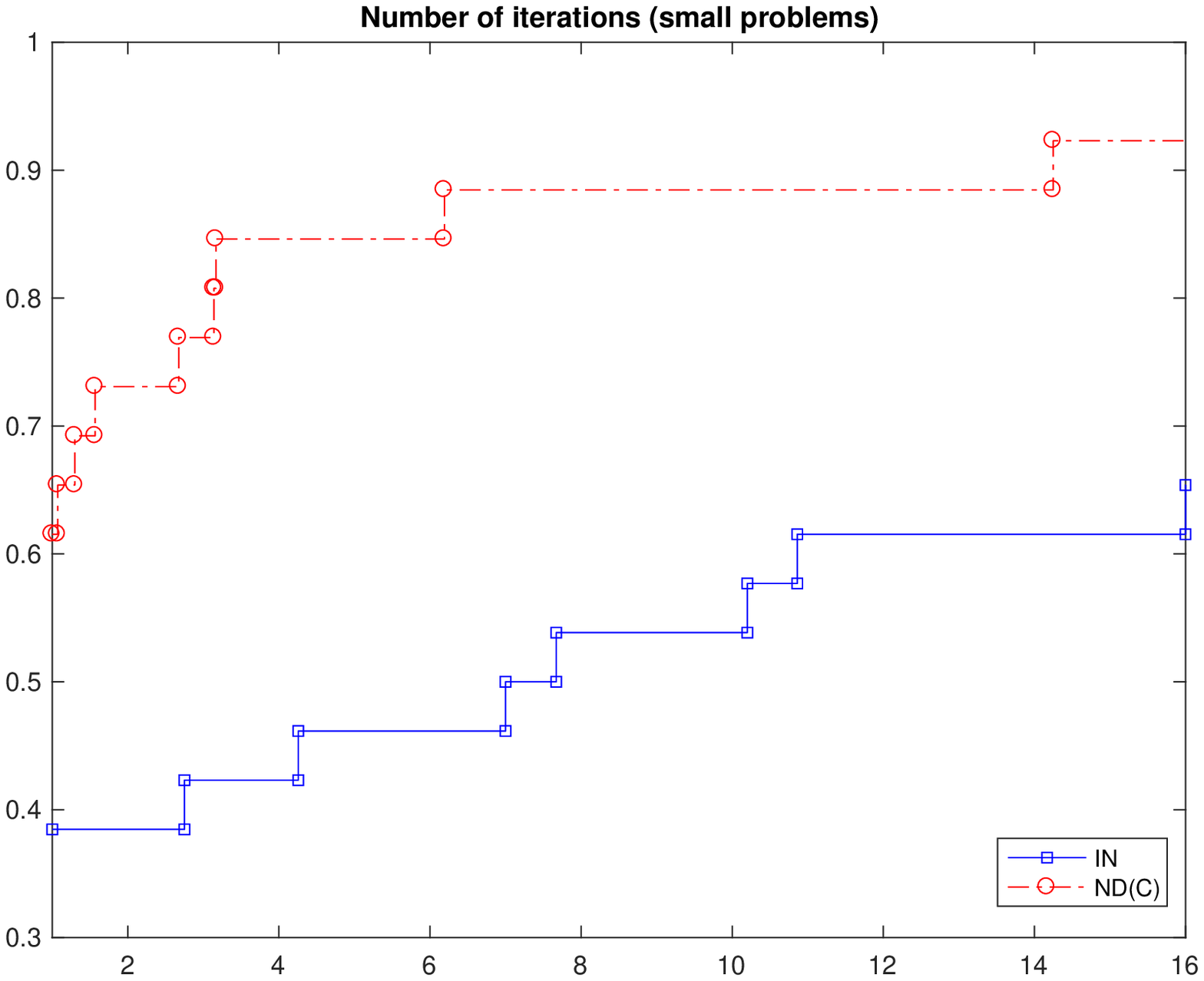}
	\end{minipage}
	\caption{Testing the Newton direction recovery within a line-search algorithm. Performance profiles for the numbers of Hessian-vector products and iterations, for the set of small problems of Appendix~\ref{APP:larger_problems}.}
	\label{PP:PP_hv_larger}
\end{figure}
%%%%%%%%%%%%%%%
\begin{figure}[H]
	\begin{minipage}[H]{0.5\linewidth}
		\centering
		\includegraphics[height=5.5cm,width=8.5cm]{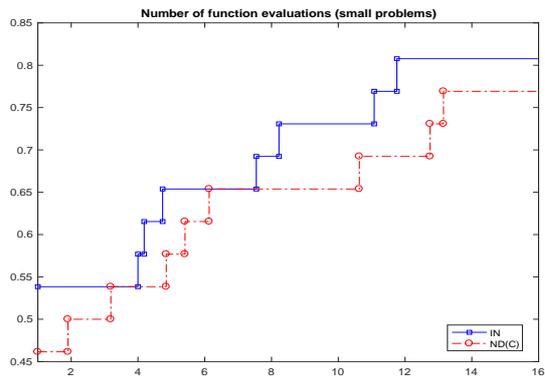}
	\end{minipage}%

	\caption{Testing the Newton direction recovery within a line-search algorithm. Performance profiles for the numbers of function evaluations for the set of small problems of Appendix~\ref{APP:larger_problems}.}
	\label{PP:PP_func_larger}
\end{figure}

\section{Final remarks} \label{sec:finalremarks}

In this paper we showed how to use interpolation techniques from Derivative-Free Optimization to model Hessian-vector products. We aimed at presenting new, refreshing ideas, laying down the theoretical groundwork for future more elaborated algorithmic developments. Two approaches were presented and analyzed. In the first one, one aims at recovering a model the Hessian matrix, possibly sparse if the true Hessian sparsity pattern is known. A drawback of this approach is that at most one Hessian-vector product can be used in the recovery. The second approach aims at directly recovering the Newton direction itself, and it may incorporate several Hessian-vector products at the same time. However, a dense system of linear equations needs to be solved.

It is left for future work the development of competitive versions of these two approaches for medium/large scale problems. In the particular case of the second approach based on the calculation of a Newton direction model, one can consider solving the linear system~(\ref{interpolation-inverse}) using
an iterative solver. In such a case, one can easily envision a parallel procedure for the storage of the matrix $M_L^z$ (storing row-wise the vectors $z^\ell$'s) and the calculation of the products $M_L^z$ times a vector required when applying an iterative solver.

A third recovery approach can also be derived, where the Newton direction and the inverse of the Hessian are recovered at once (possibly never storing the whole inverse, rather forming its product times the gradient). This approach has performed the worse, and we have decided to leave the details for a future PhD thesis of the first author.

%%%%%%%%%%%%%%%%%
\begin{appendices}
	
 \section{Performance profiles}\label{APP:PP}
 Performance profiles~\cite{EDDolan_JJMore_2002} are used to compare the performance of several solvers on a set of problems. Let $\mathcal{S}$ be a set of solvers and $\mathcal{P}$ a set of problems. Let $t_{p,s}$ be the performance metric of the solver $s\in\mathcal{S}$ on the problem $p\in\mathcal{P}$. Then the  performance profile of solver $s \in\mathcal{S}$ is defined as the fraction of problems where the performance ratio is at most $\tau$,
 $$\rho_{s}(\tau)=\frac{1}{|\mathcal{P}|}\left|\left\{p \in \mathcal{P} : \frac{t_{p, s}}{\min \left\{t_{p, s^{\prime}} : s^{\prime} \in S\right\}} \leq \tau\right\}\right|,$$
 where $|\mathcal{P}|$ denotes the cardinality of $\mathcal{P}$.
 The value of $\rho_{s}(1)$ expresses the percentage of problems on which solver~$s$ performed the best.
 The values of $\rho_{s}(\tau)$ for large $\tau$ indicate the percentage of problems successfully solved by
 solver~$s$. Hence, $\rho_{s}(1)$ and $\rho_{s}(\tau)$ for large $\tau$ are, respectively, measures of the efficiency and robustness of a given solver~$s$. Solvers with profiles above others are naturally preferred.

%%%%%%%%%%%%%%%%%
\section{Very small test problems}\label{APP:test_problems}

\begin{table}[H]
	\centering
	\setlength{\tabcolsep}{4mm} % set the length of each row
	\caption{List of 48 very small CUTEst test problems.}
	\label{table:small_problems}
	\small
	\renewcommand{\arraystretch}{1.0}
	\begin{tabular}{|cccccc|}
		\hline
		Name           &Dimension       &Name        &Dimension          &Name        &Dimension     \\
		\hline
		ALLINITU     &4              &ARGLINA       &10          &ARWHEAD        &10      \\
		BEALE           &2 		        &BIGGS6         &6             &BOX3                &3        \\
		BROWNAL   &10             &BRYBND       &10           &CHNROSNB     &10     \\
		COSINE        &10             &CUBE            &2             &DIXMAANA       &15     \\	
		DIXMAANB   &15	         &DIXMAAND   &15	       &DIXMAANE       &15 	   \\
		DIXMAANF    &15	         &DIXMAANG   &15	        &DIXMAANH      &15     \\
		DIXMAANI     &15	      &DIXMAANJ    &15            &DIXMAANK       &15	 \\
		DIXMAANL    &15	         &DIXON3DQ   &10            &DQDRTIC          &10     \\
		EDENSCH10  &10         &ENGVAL2       &3              &EXPFIT              &2     \\
		FMINSURF    &15      	 &GROWTHLS  &3              &HAIRY                &2      \\
		HATFLDD      &3            &HATFLDE      &3               &HEART8LS       &8   \\
		HELIX              &3          &HILBERTA      &10            &HILBERTB          &10 \\
		HIMMELBG    &2          &HUMPS           &2              &KOWOSB            &4      \\
		MANCINO      &30        &MSQRTALS    &4              &MSQRTBLS        &9    \\
		POWER           &10         &SINEVAL         &2               &SNAIL                  &2         \\
		SPARSINE       &10         &SPMSRTLS      &28           &TRIDIA              &10     \\
		\hline
		
	\end{tabular}
\end{table}		

\section{Small sparse test problems}\label{APP:sparse_problems}
\begin{table}[H]
\centering
\setlength{\tabcolsep}{4mm} % set the length of each row
\caption{List of 12 sparse small CUTEst test problems.}
\label{table:sparse_problems}
\small
\renewcommand{\arraystretch}{1.0}
\begin{tabular}{|cccccc|}
\hline
Name           &Dimension       &Name        &Dimension          &Name        &Dimension     \\
\hline
BDQRTIC     &10                   &BROYDN7D         &50          &COSINE        &200      \\
DQRTIC        &10                  &EDENSCH           &200         &ENGVAL1      &200        \\
LIARWHD     &100                 &NONSCOMP       &50          &PENTDI         &100    \\
SROSENBR  &50                   &TOINTGSS          &50          &TRIDIA          &200     \\
\hline
\end{tabular}
\end{table}	
%%%%%%%%%%%%%%%%%

\section{Small test problems}\label{APP:larger_problems}
\begin{table}[H]
	\centering
	\setlength{\tabcolsep}{4mm} % set the length of each row
	\caption{List of 26 small CUTEst test problems.}
	\label{table:larger_problems}
	\small
	\renewcommand{\arraystretch}{1.0}
	\begin{tabular}{|cccccc|}
		\hline
		Name           &Dimension       &Name        &Dimension          &Name        &Dimension     \\
		\hline
		BOX             &200                   &BOXPOWER &200                  &BRYBND      &100\\
		CHNROSNB &50                   &DIXON3DQ    &200                 &DQDRTIC   &100\\
		EDENSCH    &200                 &ENGVAL1        &200                 &EXTROSNB  &100\\
		GENHUMPS    &100             &HILBERTA       &200                 &HILBERTB       &200\\
		INTEQNELS     &100             &LIARWHD        &200                 &MOREBV         &200\\
		PENTDI           &100               &PENALTY1        &100                 &POWELLSG    &36\\
		SPARSINE    &100                 &SROSENBR       &50                 &SROSENBR       &100\\
		TESTQUAD    &100               &TOINTGSS       &50                  &TQUARTIC        &100\\
		TRIDIA             &200              &VAREIGVL         &100\\
		\hline
		
	\end{tabular}
\end{table}	
%%%%%%%%%%%%%%%%%

\end{appendices}

\bibliographystyle{plain}
\bibliography{ref-model-Hvp}
\end{document}